\documentclass[12pt,reqno,twoside]{amsart}
\usepackage[utf8]{inputenc}
\usepackage[T1]{fontenc}
\usepackage{lmodern}
\usepackage[a4paper]{geometry}
\usepackage{xcolor}
\usepackage{graphicx}
\usepackage{amsmath}
\usepackage{amsfonts}
\usepackage{amssymb}
\usepackage{mathtools}
\usepackage{mathrsfs}
\usepackage{cite}
\usepackage[english]{babel}
\newtheorem{theorem}{Theorem}

\newtheorem{lemma}[theorem]{Lemma}
\newtheorem{proposition}[theorem]{Proposition}
\newtheorem{definition}[theorem]{Definition}

\begin{document}
\title{Uniqueness of mass-conserving self-similar solutions to\\
Smoluchowski's coagulation equation with inverse power law kernels} 
\author{Philippe Lauren\c{c}ot}
\address{Institut de Math\'ematiques de Toulouse, UMR~5219, Universit\'e de Toulouse, CNRS, F--31062 Toulouse Cedex 9, France}
\email{laurenco@math.univ-toulouse.fr}


\maketitle

\begin{abstract}
Uniqueness of mass-conserving self-similar solutions to Smoluchowski's coagulation equation is shown when the coagulation kernel $K$ is given by $K(x,x_*)=2(x x_*)^{-\alpha}$, $(x,x_*)\in (0,\infty)^2$, for some $\alpha>0$. 
\end{abstract}

%
%
\pagestyle{myheadings}
\markboth{\sc{Ph. Lauren\c{c}ot}}{\sc{Uniqueness of mass-conserving self-similar solutions}}

\section{Introduction}\label{sec1}

Smoluchowski's coagulation equation is a mean-field model describing the growth of particles which increase their sizes by successive pairwise mergers \cite{Smol16, Smol17}. Denoting the size distribution function of particles of size $x\in (0,\infty)$ at time $t\in [0,\infty)$ by $f(t,x)$, Smoluchowski's coagulation equation reads
\begin{equation}
\partial_t f(t,x) = \frac{1}{2} \int_0^x K(x_*,x-x_*) f(t,x_*) f(t,x-x_*)\ \mathrm{d}x_* - \int_0^\infty K(x,x_*) f(t,x_*) f(t,x)\ \mathrm{d}x_* \label{a1}
\end{equation}
for $(t,x)\in (0,\infty)^2$. In \eqref{a1}, $K$ is the coagulation kernel and $K(x,x_*)=K(x_*,x)\ge 0$ measures the likelihood that a particle of size $x$ and a particle of size $x_*$ merge. More specifically, the first term in the right-hand side of \eqref{a1} accounts for the formation of particles of size $x$ resulting from the coalescence of two particles with respective sizes $x_*\in (0,x)$ and $x-x_*$. The second term in the right-hand side of \eqref{a1} describes the loss of particles of size $x$ as they coagulate with other particles of arbitrary sizes. Observe that, during the just described coagulation mechanism, there is no loss of matter and the total mass of the particles is expected to be constant throughout time evolution, that is,
\begin{equation}
M_1(f(t)) := \int_0^\infty x f(t,x)\ \mathrm{d}x = const.\ , \qquad t\ge 0\ . \label{a2}
\end{equation}
It is however well-known by now that infringement of the conservation of matter \eqref{a2} occurs for coagulation kernels $K$ growing sufficiently rapidly for large sizes, a typical example being $K(x,y)=(xy)^{\lambda/2}$ when $\lambda>1$ \cite{EMP02,Jeon98,Leyv83,Leyv03,LeTs81a,Ziff80}. This phenomenon is usually referred to as \textsl{gelation} and corresponds to a runaway growth in the coagulation process: the rapid growth of the coagulation kernel for large sizes enhances the formation of larger and larger particles and leads to the appearance of particles of infinite size (also called \textsl{giant particle}) in finite time. Since the size  distribution function $f$ only accounts for finite size particles, there is thus an escape of matter from the system of finite size particles towards giant particles.

When the conservation of matter \eqref{a2} holds true for all times, a key issue in the study of Smoluchowski's coagulation equation \eqref{a1} is its predictive behaviour, especially regarding the long term dynamics. In that direction, a commonly accepted conjecture is that the size distribution function $f$ displays a scaling behaviour as $t\to\infty$ which complies with conservation of matter, namely,
\begin{equation}
f(t,x)\sim \frac{1}{\sigma(t)^2} \varphi\left( \frac{x}{\sigma(t)} \right) \;\text{ as }\; t\to\infty\ , \label{a3}
\end{equation}
where $\sigma(t)$ denotes the mean size at time $t$ and $\varphi$ the scaling or self-similar profile which is expected to be non-negative and to have a finite mass, both being yet undetermined \cite{Leyv03,vDEr88}. The assertion \eqref{a3} is usually referred to as the \textsl{dynamical scaling hypothesis} in the literature and its validity turns out to be one of the main issues in the analysis of Smoluchowski's coagulation equation. Current research focuses on homogeneous coagulation kernels $K$ satisfying
\begin{equation}
K(\xi x,\xi x_*) = \xi^\lambda K(x,x_*)\ , \qquad (\xi,x,x_*)\in (0,\infty)^3\ , \label{a4}
\end{equation}
for some $\lambda\in (-\infty,1)$ (we leave aside the case $\lambda=1$ which is rather peculiar, see \cite{Bert02a,BNV16,HNV17,Leyv03,MePe04,vDEr88}, while $\lambda>1$ corresponds to the onset of gelation and thus a different dynamics \cite{Leyv03,vDEr88}). When $\lambda\in (-\infty,1)$, it is expected that 
\begin{equation}
(t,x)\mapsto \frac{1}{\sigma(t)^2} \varphi\left( \frac{x}{\sigma(t)} \right) \label{a5}
\end{equation}
is a self-similar solution to the coagulation equation \eqref{a1}. Inserting the ansatz \eqref{a5} in \eqref{a1} and using \eqref{a4} lead to the existence of a positive constant $w>0$ such that
\begin{equation}
\sigma(t) = (1+w(1-\lambda)t)^{1/(1-\lambda)}\ , \qquad t\ge 0\ ,  \label{a6}
\end{equation}
and
\begin{align}
w \left( x \partial_x \varphi(x) + 2 \varphi(x) \right) & + \frac{1}{2} \int_0^x K(x_*,x-x_*) \varphi(x_*) \varphi(x-x_*)\ \mathrm{d}x_* \nonumber\\
& \phantom{0123456789} - \int_0^\infty K(x,x_*) \varphi(x_*) \varphi(x)\ \mathrm{d}x_* = 0\ , \qquad x\in (0,\infty)\ , \label{a7}
\end{align}
thereby determining the mean size $\sigma$  (up to the constant $w$) and providing a nonlinear and nonlocal integrodifferential equation solved by the scaling profile $\varphi$. We are then left with investigating the existence of non-negative solutions $\varphi$ to \eqref{a7} with a prescribed total mass $M_1(\varphi)=\varrho$ for some given $\varrho>0$. When $K\equiv 2$, a family of explicit solutions $x\mapsto (w^2/\varrho)e^{-wx/\varrho}$ is known for a long time  and several existence results have been obtained in the past decade for various coagulation kernels \cite{BLLxx,BNV17,EMRR05,FoLa05,NiVe14a}. Special attention is also paid to the identification of the behaviour of the scaling profile for small and large sizes \cite{CaMi11,EsMi06,FoLa06a,NiVe11a,NiVe14a}, see also \cite{Leyv03,MNV11,vDEr88} for formal asymptotic expansions. 

Though not exhausted, the question of the existence of scaling profiles is thus answered in a satisfactory way by now. Nevertheless, we are still far away from a complete proof of the dynamical scaling hypothesis \eqref{a3} and a first step towards its validity is the uniqueness (up to scaling) of the scaling profile. The first result in that direction deals with the constant coagulation kernel  $K\equiv 2$ for which it is known that, given  $w>0$ and $\varrho>0$, there is a unique non-negative scaling profile $\varphi_{w,\varrho}$ solving \eqref{a7} and satisfying $M_1(\varphi_{w,\varrho})=\varrho$ \cite{MePe04}. It is actually given explicitly by $\varphi_{w,\varrho}(x) = (w^2/\varrho)e^{-wx/\varrho}$, $x\in (0,\infty)$, and the uniqueness proof relies on the fact that the Bernstein (or desingularized Laplace) transform $\mathcal{B}\varphi$ of any scaling profile $\varphi$ defined by 
\begin{equation*}
\mathcal{B}\varphi(\xi) := \int_0^\infty \left( 1 - e^{-x\xi} \right) \varphi(x)\ \mathrm{d}x\ , \qquad \xi\in (0,\infty)\ ,
\end{equation*}
is a solution to the ordinary differential equation
\begin{equation*}
w \xi \partial_\xi\mathcal{B}\varphi(\xi) + [\mathcal{B}\varphi(\xi)]^2 - w \mathcal{B}\varphi(\xi) = 0\ , \qquad \xi\in (0,\infty)\ , 
\end{equation*}
which can be solved explicitly. Though such a nice device does not extend to other coagulation kernels, the Laplace transform is also employed in \cite{NTV16b,NiVe14b} to prove the uniqueness of scaling profiles for coagulation kernels $K$ with homogeneity zero which are sufficiently small perturbations of the constant kernel. The coagulation kernels $K$ dealt with in \cite{NTV16b,NiVe14b} satisfy in particular 
\begin{equation*}
2 \le K(x,x_*) \le 2 + \varepsilon \left[ \left( \frac{x}{x_*}\right)^\alpha + \left( \frac{x_*}{x}\right)^\alpha \right]\ , \qquad (x,x_*)\in (0,\infty)^2\ ,
\end{equation*}
for $\alpha\in [0,1/2)$ and $\varepsilon>0$ sufficiently small. Let us also mention that some partial uniqueness results are also obtained in \cite{CaMi11} for the coagulation kernel $K(x,x_*) =x^\alpha x_*^\beta + x^\beta x_*^\alpha$ when $\alpha\in (-1,0]$, $\beta\in [\alpha,1)$, and $\lambda:=\alpha+\beta\in (-1,1)$. More precisely, it is shown in \cite{CaMi11} that prescribing the moments of order one and $\lambda$ guarantees the uniqueness of the scaling profile when $\alpha=0$. In the same vein, when $\alpha<0$, a similar result is available when prescribing not only the moments of order one, $\alpha$, and $\beta$, but also the behaviour for small sizes. 

The purpose of this note is to contribute to the uniqueness issue of scaling profiles for the particular class of coagulation kernels
\begin{equation}
K(x,x_*) := \frac{2}{(xx_*)^{\alpha}}\ , \qquad x\in (0,\infty)\ , \label{a8}
\end{equation}
which is introduced in \cite{ClKa99}, the parameter $\alpha$ being a positive real number. More precisely, the main result of this note reads:

\begin{theorem}\label{thma1}
Let $w>0$ and $\varrho>0$ and assume that the coagulation kernel $K$ is given by \eqref{a8} for some $\alpha>0$. There is a unique non-negative scaling profile $\varphi_{w,\varrho}$ satisfying
\begin{equation}
\varphi_{w,\varrho}\in \mathcal{C}^1(0,\infty) \cap \bigcap _{m\in\mathbb{R}} L^1(0,\infty,x^m\mathrm{d}x)\ , \qquad M_1(\varphi_{w,\varrho}) = \varrho\ , \label{a9}
\end{equation}
and \eqref{a7} for all $x\in (0,\infty)$.
\end{theorem}

Let us first recall that the existence of a scaling profile solving \eqref{a7} and enjoying the regularity and integrability properties \eqref{a9} can be shown by adapting arguments from \cite{EMRR05,FoLa05}, see \cite{BLLxx} and Proposition~\ref{propb1} below. Concerning uniqueness, the proof takes advantage of the specific structure of the coagulation kernel \eqref{a8}, as did those developed in \cite{MePe04,NTV16b,NiVe14b} for other kernels, but, instead of using the Laplace transform, we take a different route and use a weighted indefinite integral of $\varphi$ in the spirit of \cite{FoLa06b}. More precisely, given $w>0$ and a scaling profile $\varphi$ satisfying \eqref{a7}, we introduce 
\begin{equation*}
H(x) := \int_x^\infty \frac{\varphi(x_*)}{x_*^\alpha}\ \mathrm{d}x_*\ , \qquad x\in (0,\infty)\ ,
\end{equation*}
and derive an equation solved by $H$. We then show that this equation has a unique solution satisfying $H(0)=1$, the latter property corresponding to the additional assumption $M_{-\alpha}(\varphi)=1$. We finally use a scaling argument to connect the mass constraint $M_1(\varphi)=\varrho$ to $M_{-\alpha}(\varphi)=1$. 

As a final comment, while Theorem~\ref{thma1} provides a step further towards the validity of the dynamical scaling hypothesis \eqref{a3} for the coagulation kernel \eqref{a8}, the approach developed in this note does not seem to provide valuable information on the time-dependent problem. It is thus likely that the stability of $\varphi_{w,\varrho}$ (if true) requires additional ideas. In fact, as far as we know, the constant coagulation kernel $K\equiv 2$ is the only one with homogeneity strictly smaller than one for which the validity of \eqref{a3} is rigorously established \cite{KrPe94,MePe04}, the additive coagulation kernel $K(x,x_*)=x+x_*$ being also handled in \cite{MePe04}.  

Throughout the paper we use the following notation: given $m\in \mathbb{R}$, we set $X_m := L^1(0,\infty,x^m\mathrm{d}x)$ and put
\begin{equation*}
M_m(f) := \int_0^\infty x^m f(x)\ \mathrm{d}x \ , \qquad f\in L^1(0,\infty,x^m\mathrm{d}x)\ .
\end{equation*}
We also denote the positive cone of $X_m$ by $X_m^+$, that is, $X_m^+ := \left\{ f\in X_m\ :\ f\ge 0 \;\text{ a.e. in }\; (0,\infty) \right\}$. 

From now on, we fix $w>0$ and $\alpha>0$ and assume that the coagulation kernel $K$ is given by \eqref{a8}. 

\section{Scaling Profiles}\label{sec2}

We begin with the definition of a scaling profile solving \eqref{a7} for the particular choice \eqref{a8} of the coagulation kernel.

\begin{definition}\label{defb0}
A scaling profile is a function $\varphi\in X_1^+\cap \mathcal{C}^1(0,\infty)$ such that
\begin{equation}
\varphi\in \bigcap _{m\in\mathbb{R}} X_m\ , \label{b0}
\end{equation}
which satisfies \eqref{a7} for all $x\in (0,\infty)$. 
\end{definition}

We next collect some properties of scaling profiles and first state an existence result.

\begin{proposition}\label{propb1}
Let $\varrho>0$. There is at least one scaling profile $\varphi$ in the sense of Definition~\ref{defb0} such that $M_1(\varphi) = \varrho$. In addition,
\begin{equation}
w x_0^2 \varphi(x_0) = 2 \int_0^{x_0} \int_{x_0-x}^\infty x^{1-\alpha} x_*^{-\alpha} \varphi(x)\varphi(x_*)\ \mathrm{d}x_*\mathrm{d}x\ , \qquad x_0\in (0,\infty)\ . \label{b00} 
\end{equation}
\end{proposition}

\begin{proof}
The existence of a (weak) solution $\varphi\in X_1^+$ to \eqref{a7} which satisfies \eqref{b0}, \eqref{b00}, and  $M_1(\varphi) = \varrho$ can be performed by adapting suitably the arguments developed in \cite{EMRR05,FoLa05} for related coagulation kernels, see \cite{BLLxx} for a complete proof. The $C^1$-smoothness of $\varphi$ next follows from \eqref{b00} according to \cite{CaMi11} and, together with the already established integrability properties, implies that $\varphi$ solves \eqref{a7} pointwisely.
\end{proof}

We next report a scaling invariance of scaling profiles which stems from the homogeneity of the coagulation kernel $K$.

\begin{lemma}\label{lemb1}
Let $\varphi$ be a scaling profile in the sense of Definition~\ref{defb0} and $a>0$. Then  the function $\varphi_a$ defined by $\varphi_a(x) := a^{1-2\alpha} \varphi(ax)$, $x\in (0,\infty)$, is also a scaling profile in the sense of Definition~\ref{defb0} and
\begin{equation*}
M_m(\varphi_a) = a^{-m-2\alpha} M_m(\varphi)\ , \qquad m\in\mathbb{R}\ .
\end{equation*}
\end{lemma}

The specific form \eqref{a8} of the coagulation kernel also entails an additional property of scaling profiles.

\begin{lemma}\label{lemb2}
Let $\varphi$ be a scaling profile in the sense of Definition~\ref{defb0}. Then
\begin{equation*}
w M_0(\varphi) = M_{-\alpha}(\varphi)^2\ .
\end{equation*}
\end{lemma}

\begin{proof}
We integrate \eqref{a7} over $(0,\infty)$ and use Fubini's theorem to obtain the claim.
\end{proof}

We next introduce an auxiliary function and identify the equation it solves.

\begin{lemma}\label{lemb3}
Let $\varphi$ be a scaling profile in the sense of Definition~\ref{defb0} and define
\begin{equation}
h(x) := \frac{\varphi(x)}{x^\alpha} \;\text{ and }\; H(x) := \int_x^\infty h(x_*)\ \mathrm{d}x_*\ , \qquad x\in (0,\infty)\ . \label{b1}
\end{equation}
Then 
\begin{equation}
M_{\alpha-1}(H) = \frac{M_0(\varphi)}{\alpha} = \frac{M_{-\alpha}(\varphi)^2}{\alpha w}\ , \label{b99}
\end{equation}
and $H$ solves
\begin{equation}
w x^{1+\alpha} \partial_x H(x) + w x^\alpha H(x) + \alpha w \int_x^\infty x_*^{\alpha-1} H(x_*)\ \mathrm{d}x_* + h*H(x) - M_0(h) H(x) = 0 \label{b2}
\end{equation}
for $x\in (0,\infty)$, where $*$ denotes the convolution product, that is,
\begin{equation*}
f*g(x) := \int_0^x f(x_*) g(x-x_*)\ \mathrm{d}x_*\ , \qquad x\in (0,\infty)\ .
\end{equation*}
\end{lemma}

\begin{proof}
First, \eqref{b99} readily follows from \eqref{b0}, \eqref{b1}, Fubini's theorem., and Lemma~\ref{lemb2}. Next, since $\varphi(x)=x^\alpha h(x)$ for $x\in (0,\infty)$ by \eqref{b1}, we infer from \eqref{a7} that $h$ solves
\begin{equation}
w x^{1+\alpha} \partial_x h(x) + w(\alpha+2) x^\alpha h(x) + h*h(x) - 2 M_0(h) h(x) = 0\ , \qquad x\in (0,\infty)\ . \label{b3}
\end{equation}
Let $x_0\in (0,\infty)$. Integrating \eqref{b3} with respect to $x$ over $(x_0,\infty)$ gives
\begin{equation}
w \left[ x^{\alpha+1} h(x) \right]_{x=x_0}^{x=\infty} + w \int_{x_0}^\infty x^\alpha h(x)\ \mathrm{d}x + \int_{x_0}^\infty h*h(x)\ \mathrm{d}x - 2 M_0(h) H(x_0) = 0\ . \label{b4}
\end{equation}
On the one hand, it follows from \eqref{b00} and \eqref{b1} that, for $x>0$,
\begin{align*}
w x^{1+\alpha} h(x) = w x \varphi(x) & = \frac{2}{x} \int_0^x \int_{x-y}^\infty y^{1-\alpha} z^{-\alpha} \varphi(y)\varphi(z)\ \mathrm{d}z\mathrm{d}z \\ 
& \le \frac{2}{x} M_{1-\alpha}(\varphi) M_{-\alpha}(\varphi)\ ,
\end{align*}
hence
\begin{equation}
\lim_{x\to \infty} x^{1+\alpha} h(x) = 0\ . \label{b5}
\end{equation}
On the other hand, Fubini's theorem ensures that
\begin{equation}
\int_{x_0}^\infty h*h(x)\ \mathrm{d}x = h*H(x_0) + M_0(h) H(x_0)\ . \label{b6}
\end{equation}
Combining \eqref{b1}, \eqref{b4}, \eqref{b5}, and \eqref{b6}, we end up with
\begin{equation*}
w x_0^{1+\alpha} \partial_x H(x_0) - w \left[ x^\alpha H(x) \right]_{x=x_0}^{x=\infty} + w \alpha \int_{x_0}^\infty x^{\alpha-1} H(x)\ \mathrm{d}x + h*H(x_0) - M_0(h) H(x_0) = 0\ .
\end{equation*}
Since
\begin{equation*}
x^\alpha H(x) \le \int_x^\infty x_*^\alpha h(x_*)\ \mathrm{d}x_* = \int_x^\infty \varphi(x_*)\ \mathrm{d}x_*
\end{equation*}
and $\varphi\in X_0$, we deduce that $x^\alpha H(x)\to 0$ as $x\to\infty$ and thereby complete the proof.
\end{proof}

\section{Uniqueness}\label{sec3}

The heart of the proof of Theorem~\ref{thma1} is the following uniqueness result.

\begin{proposition}\label{propc1}
Let $\varphi_1$ and $\varphi_2$ be two scaling profiles in the sense of Definition~\ref{defb0} and assume further that
\begin{equation}
M_{-\alpha}(\varphi_1) = M_{-\alpha}(\varphi_2) = 1\ . \label{c1}
\end{equation}
Then $\varphi_1 = \varphi_2$.
\end{proposition}

\begin{proof}
For $i\in\{1,2\}$, we define
\begin{equation}
h_i(x) := \frac{\varphi_i(x)}{x^\alpha}\ , \qquad H_i(x) := \int_x^\infty h_i(x_*)\ \mathrm{d}x_*\ , \qquad x\in (0,\infty)\ , \label{c99}
\end{equation}
and observe that \eqref{c1} implies that
\begin{equation}
M_0(h_i) = H_i(0) = 1\ , \qquad i\in\{1,2\}\ . \label{c2}
\end{equation}
Introducing 
\begin{equation}
E := H_1 - H_2 \ , \qquad \Sigma := \mathrm{sign}(E)\ , \label{c3} 
\end{equation}
we infer from \eqref{b2} in Lemma~\ref{lemb3} and \eqref{c2} that $E$ solves
\begin{align}
w x^{1 + \alpha} \partial_x E(x) + w x^\alpha E(x) & + \alpha w \int_x^\infty x_*^{\alpha-1} E(x_*)\ \mathrm{d}x_* \nonumber \\
& \phantom{0123456789} + \left( h_1*H_1 - h_2*H_2 \right)(x) - E(x) = 0\ . \label{c4}
\end{align}
Since
\begin{equation*}
h_1*H_1 - h_2*H_2 = \frac{1}{2} \left[ (h_1-h_2)*(H_1+H_2) + (h_1+h_2)*E \right]\ , 
\end{equation*}
and
\begin{align*}
(h_1-h_2)*(H_1+H_2) & = (h_1+h_2)*E + E(0) (H_1+H_2) - (H_1+H_2)(0) E \\
& = (h_1+h_2)*E - 2E
\end{align*}
by \eqref{c2}, we realize that $h_1*H_1 - h_2*H_2 = (h_1+h_2)*E - E$. Inserting this formula in \eqref{c4}, we end up with
\begin{equation}
w x^{1 + \alpha} \partial_x E(x) + w x^\alpha E(x) + \alpha w \int_x^\infty x_*^{\alpha-1} E(x_*)\ \mathrm{d}x_* + (h_1+h_2)*E(x) = 2E(x) \label{c6}
\end{equation}
for $x\in (0,\infty)$. We multiply \eqref{c6} by $\Sigma(x)$ and integrate over $(0,\infty)$ to obtain
\begin{align}
2 \int_0^\infty |E(x)|\ \mathrm{d}x & = w \int_0^\infty x^{1+\alpha} (\partial_x |E|)(x)\ \mathrm{d}x + w \int_0^\infty x^\alpha |E(x)|\ \mathrm{d}x \nonumber \\
& \qquad + \alpha w \int_0^\infty \Sigma(x) \int_x^\infty x_*^{\alpha-1} E(x_*)\ \mathrm{d}x_*\mathrm{d}x + \int_0^\infty \Sigma(x) (h_1+h_2)*E(x)\ \mathrm{d}x\ . \label{c7}
\end{align}
Now,
\begin{equation}
\int_0^\infty x^{1+\alpha} (\partial_x |E|)(x)\ \mathrm{d}x = \left[ x^{1+\alpha} |E(x)| \right]_{x=0}^{x=\infty} - (1+\alpha) \int_0^\infty x^\alpha |E(x)|\ \mathrm{d}x\ . \label{c8}
\end{equation}
On the one hand, we infer from \eqref{c2} and the positivity of $\alpha$ that $x^{1+\alpha}E(x)\to 0$ as $x\to 0$. On the other hand, according to \eqref{c99} and the integrability properties \eqref{b0} of $\varphi_1$ and $\varphi_2$,
\begin{equation*}
x^{1+\alpha} |E(x)| \le x^{1+\alpha} (H_1+H_2)(x) \le \int_x^\infty x(\varphi_1+\varphi_2)(x)\ \mathrm{d}x \mathop{\longrightarrow}_{x\to\infty} 0 \ .
\end{equation*}
Consequently, the first term in the right-hand side of \eqref{c8} vanishes and it follows from \eqref{c7}, \eqref{c8},  and Fubini's theorem that
\begin{align}
& 2 \int_0^\infty |E(x)|\ \mathrm{d}x - \int_0^\infty \Sigma(x) (h_1+h_2)*E(x)\ \mathrm{d}x \nonumber\\
& \qquad = - \alpha w \int_0^\infty x^\alpha |E(x)|\ \mathrm{d}x + \alpha w \int_0^\infty x^{\alpha-1} E(x) \int_0^x \Sigma(x_*)\ \mathrm{d}x_*\mathrm{d}x\ . \label{c9}
\end{align}
Owing to \eqref{c2} and the property $|\Sigma|\le 1$, we notice that
\begin{align*}
\int_0^\infty \Sigma(x) (h_1+h_2)*E(x)\ \mathrm{d}x & \le \int_0^\infty (h_1+h_2)*|E|(x)\ \mathrm{d}x \\
& = \left[ M_0(h_1)+M_0(h_2) \right] \int_0^\infty |E(x)|\ \mathrm{d}x \\
& = 2 \int_0^\infty |E(x)|\ \mathrm{d}x \ ,
\end{align*}
and
\begin{align*}
\int_0^\infty x^{\alpha-1} E(x) \int_0^x \Sigma(x_*)\ \mathrm{d}x_*\mathrm{d}x \le \int_0^\infty x^\alpha |E(x)|\ \mathrm{d}x\ .
\end{align*}
Consequently, the left-hand side of \eqref{c9} is non-negative while its right-hand side is non-positive, so that they both vanish, hence
\begin{align}
2 \int_0^\infty |E(x)|\ \mathrm{d}x & = \int_0^\infty \Sigma(x) (h_1+h_2)*E(x)\ \mathrm{d}x\ , \label{c10} \\
\int_0^\infty x^\alpha |E(x)|\ \mathrm{d}x & = \int_0^\infty x^{\alpha-1} E(x) \int_0^x \Sigma(x_*)\ \mathrm{d}x_*\mathrm{d}x\ . \label{c11}
\end{align}
Since $E \Sigma^2 = E$, an alternative formulation for \eqref{c11} reads
\begin{equation}
\int_0^\infty x^{\alpha-1} |E(x)| \left( \int_0^x \left[ 1 - \Sigma(x) \Sigma(x_*) \right]\ \mathrm{d}x_* \right) \mathrm{d}x = 0\ . \label{c12}
\end{equation}
We next define 
\begin{equation*}
\mathcal{P} := \{ x\in (0,\infty)\ :\ E(x)>0 \} \;\text{ and }\; \mathcal{N} := \{ x\in (0,\infty)\ :\ E(x)< 0\}\ .
\end{equation*}
Since $E$ clearly belong to $C(0,\infty)$, both $\mathcal{P}$ and $\mathcal{N}$ are open subsets of $(0,\infty)$. Recalling that $1\ge \Sigma(x)\Sigma(x_*)$ for all $(x,x_*)\in (0,\infty)^2$, we infer from \eqref{c12} that
\begin{equation*}
0 \ge \int_{\mathcal{P}} x^{\alpha-1} |E(x)| \left( \int_0^x \left[ 1 - \Sigma(x_*) \right]\ \mathrm{d}x_* \right) \mathrm{d}x \ge 2 \int_{\mathcal{P}} x^{\alpha-1} |E(x)| |\mathcal{N}\cap (0,x)|\ \mathrm{d}x\ ,
\end{equation*}
that is, $|\mathcal{N}\cap (0,x)|=0$ for almost every $x\in \mathcal{P}$. A similar argument ensures that $|\mathcal{P}\cap (0,x)|=0$ for almost every $x\in \mathcal{N}$. Therefore, by Fubini's theorem,
\begin{align*}
|\mathcal{P}| |\mathcal{N}| & = \int_0^\infty \int_0^\infty \mathbf{1}_{\mathcal{P}}(x) \mathbf{1}_{\mathcal{N}}(x_*)\ \mathrm{d}x_*\mathrm{d}x \\
& = \int_0^\infty \int_0^x \mathbf{1}_{\mathcal{P}}(x) \mathbf{1}_{\mathcal{N}}(x_*)\ \mathrm{d}x_*\mathrm{d}x + \int_0^\infty \int_x^\infty \mathbf{1}_{\mathcal{P}}(x) \mathbf{1}_{\mathcal{N}}(x_*)\ \mathrm{d}x_*\mathrm{d}x \\
& = \int_0^\infty \mathbf{1}_{\mathcal{P}}(x) |\mathcal{N}\cap (0,x)|\ \mathrm{d}x + \int_0^\infty \int_0^{x_*} \mathbf{1}_{\mathcal{P}}(x) \mathbf{1}_{\mathcal{N}}(x_*)\ \mathrm{d}x\mathrm{d}x_* \\
& = \int_0^\infty \mathbf{1}_{\mathcal{N}}(x_*) |\mathcal{P}\cap (0,x_*)|\ \mathrm{d}x_* = 0\ .
\end{align*}
Recalling the continuity of $E$, we have thus established that either $\mathcal{P}$ or $\mathcal{N}$ is empty, so that $E$ has a constant sign on $(0,\infty)$. However, \eqref{b99} and \eqref{c1} imply that
\begin{equation*}
\int_0^\infty x^{\alpha-1} E(x)\ \mathrm{d}x = M_{\alpha-1}(H_1) - M_{\alpha-1}(H_2) = 0\ ,
\end{equation*}
from which we conclude that $E=0$, hence $\varphi_1=\varphi_2$. 
\end{proof}

\begin{proof}[Proof of Theorem~\ref{thma1}]
Let $\varrho>0$ and consider two scaling profiles $\varphi_1$ and $\varphi_2$ in the sense of Definition~\ref{defb0} satisfying in addition 
\begin{equation}
M_1(\varphi_1) = M_1(\varphi_2) = \varrho\ . \label{c13}
\end{equation} 
Introducing $a_i := M_{-\alpha}(\varphi_i)^{1/\alpha}$ for $i\in\{1,2\}$, we infer from Lemma~\ref{lemb1} that $\varphi_{i,a_i}$ defined by $\varphi_{i,a_i}(x) := a_i^{1-2\alpha} \varphi_i(a_i x)$, $x\in (0,\infty)$, is a scaling profile in the sense of Definition~\ref{defb0} such that $M_{-\alpha}(\varphi_{i,a_i})=1$. Thanks to these properties, we are in a position to apply Proposition~\ref{propc1} to conclude that $\varphi_{1,a_1}=\varphi_{2,a_2}$, that is,
\begin{equation}
a_1^{1-2\alpha} \varphi_1(a_1 x) = a_2^{1-2\alpha} \varphi_2(a_2 x)\ , \qquad x\in (0,\infty)\ . \label{c14}
\end{equation}
Multiplying both sides of the previous identity by $x$ and integrating over $(0,\infty)$, we deduce from \eqref{c13} that
\begin{equation*}
a_1^{-1-2\alpha} \varrho = a_1^{1-2\alpha} \int_0^\infty x \varphi_1(a_1 x)\ \mathrm{d}x = a_2^{1-2\alpha} \int_0^\infty \varphi_2(a_2 x)\ \mathrm{d}x = a_2^{-1-2\alpha} \varrho\ .
\end{equation*}
Consequently, $a_1=a_2$ which, together with \eqref{c14}, entails that $\varphi_1=\varphi_2$ and completes the proof.
\end{proof}


\bibliographystyle{siam}
\bibliography{UniqSelfSim}

\end{document}